\documentclass[letterpaper, reqno,11pt]{article}
\usepackage[margin=1.0in]{geometry}
\usepackage{color,latexsym,amsmath,amssymb,amsthm}
\newcommand{\CC}{\mathbb{C}}
\newcommand{\RR}{\mathbb{R}}
\newcommand{\FF}{\mathbb{F}}
\newcommand{\eps}{\varepsilon}

\newtheorem{thm}{Theorem}
\newtheorem{lem}[thm]{Lemma}
\newtheorem{cor}[thm]{Corollary}
\newtheorem{defn}[thm]{Definition}
\newtheorem{rem}[thm]{Remark}
\newtheorem{conj}[thm]{Conjecture}

\begin{document}
\pagenumbering{arabic}
\title{Counting higher order tangencies for plane curves}
\author{Joshua Zahl\thanks{University of British Columbia, Vancouver BC. Supported by a NSERC Discovery Grant.}}
\date{\today}
\maketitle
 \begin{abstract}
 We prove that $n$ plane algebraic curves determine $O(n^{(k+2)/(k+1)})$ points of $k$--th order tangency. This generalizes an earlier result of Ellenberg, Solymosi, and Zahl on the number of (first order) tangencies determined by $n$ plane algebraic curves.
 \end{abstract} 
 \section{Introduction}
In \cite{ESZ}, Ellenberg, Solymosi, and the author proved that $n$ plane algebraic curves determine $O(n^{3/2})$ points of tangency. In this paper, we will consider the question of higher-order tangencies. We will show that $n$ plane algebraic curves determine $O(n^{\frac{k+2}{k+1}})$ points of $k$--th order tangency. 

Before stating our result, we must precisely define what it means for two curves to have $k$--th order tangency. Recall that a complex plane algebraic curve is a set of the form 
\begin{equation}\label{defnPlaneCurve}
\gamma=\{(x,y)\in\CC^2\colon P(x,y)=0\},
\end{equation}
where $P\in\mathbb{C}[x,y]$ is a non-zero polynomial. If $\gamma$ is a plane algebraic curve, we can always express $\gamma$ in the form \eqref{defnPlaneCurve} where $P$ is square-free. In this case we say that the degree of $\gamma$ is equal to the degree of $P$, and a point $p\in\gamma$ is called smooth if $\nabla P(p)\neq 0$.

\begin{defn}
Let $k\geq 1$ be an integer. Let $\gamma,\tilde\gamma$ be algebraic curves in $\CC^2$. Let $p$ be a smooth point of both $\gamma$ and of $\tilde\gamma$. Applying a rotation if necessary, we can assume that neither $\gamma$ nor $\tilde\gamma$ have vertical tangent at $p$. In a neighborhood of $p$, we will parameterize $\gamma$ as $(t,h(t))$ and $\tilde\gamma$ as $(t,\tilde h(t))$. We say that $\gamma$ and $\tilde\gamma$ are tangent at $p=(x,y)$ to order $\geq k$ if $|h(t)-\tilde h(t)| = O(|t-x|^k)$ as $t\to x$.
\end{defn}

\begin{defn}\label{countOfKTangencies}
Let $\mathcal{C}$ be a set of irreducible algebraic curves in $\CC^2$ and let $k\geq 1$ be an integer. For each $p\in\CC^2$, define
$$
m_{k,\mathcal{C}}(p) = |\{\gamma\in\mathcal{C}\colon \textrm{there exists}\ \tilde\gamma\in\mathcal{C}\ \textrm{with}\ \tilde\gamma\neq\gamma\ \textrm{so that}\ \gamma\ \textrm{and}\ \tilde\gamma\ \textrm{are tangent to order}\ \geq k\ \textrm{at}\ p\}|.
$$ 
\end{defn}

Our main result is the following bound on the number of $k$--th order tangencies. Here and throughout, all implicit constants may depend on $k$ (the order of tangency) and on the maximum degree of the curves. 
\begin{thm}\label{mainThm}
Let $\mathcal{C}$ be a set of $n$ irreducible algebraic curves in $\CC^2$ and let $k\geq 1$ be an integer. Then
\begin{equation}\label{boundOnNumberTangencies}
\sum_{p\in\CC^2}m_{k,\mathcal{C}}(p)=O(n^{\frac{k+2}{k+1}}).
\end{equation}
\end{thm}

\begin{rem}
Our definition of tangency only applies to smooth points. However, since a degree $D$ plane curve has at most $\binom{D-1}{2}$ singular points, any collection of $n$ plane curves collectively have $O(n)$ singular points.
\end{rem}

Theorem \ref{mainThm} is motivated by several open problems in combinatorics, harmonic analysis, and geometric measure theory. We shall discuss these connections below.
\medskip

\noindent {\bf Curve cutting and incidence geometry}.

In \cite{SZ}, Sharir and the author showed that $n$ algebraic curves in $\RR^2$ can be cut into $O(n^{3/2}\operatorname{polylog}(n))$ Jordan arcs, so that each pair of arcs intersect at most once. This cutting technique was then used to prove new bounds in incidence geometry. Bounding the number of (first order) curve tangencies was an important model problem for curve cutting, since an arrangement of curves can be slightly perturbed so that pairs of tangent curves become pairs of curves that intersect in two nearby places, and each of these perturbed tangencies necessitates a cut. Indeed, the curve cutting arguments in \cite{SZ} built off the tangency bound arguments developed by Ellenberg, Solymosi, and the author from \cite{ESZ}. This motivates the following conjecture.

\begin{conj}\label{curveCuttingConj}
Let $\mathcal{C}$ be a set of $n$ irreducible algebraic curves in $\RR^2$. Then the curves in $\mathcal{C}$ can be cut into  $O(n^{(k+2)/(k+1)+\eps})$ Jordan arcs, so that each pair of arcs intersect at most $k$ times.
\end{conj}

Bounding the number of $k$-th order curve tangencies is a simpler model for the problem of cutting curves, since an arrangement of curves can be slightly perturbed so that pairs of curves that have $k$-th order tangency become pairs of curves that intersect at $k+1$ nearby places, and thus each of these perturbed tangencies would require a cut. If Conjecture \ref{curveCuttingConj} is true, it would lead to progress for a number of questions in incidence geometry. Chief among these is the following difficult open problem in plane incidence geometry.

\begin{conj}\label{incidenceConj}
Let $\mathcal{P}\subset\RR^2$ be a set of $n$ points and let $\mathcal{C}$ be a set of $n$ irreducible algebraic plane curves. Then the number of point-curve incidences is $O(n^{4/3})$.
\end{conj}

To date, there has been no progress on Conjecture \ref{incidenceConj} beyond the trivial bound $O(n^{3/2})$, which follows from the Cauchy-Schwarz inequality and the observation that each pair of curves intersect in $O(1)$ places. However, if the $k=2$ case of Conjecture \ref{curveCuttingConj} were true, then by the crossing lemma (see i.e. \cite[Theorem 8.2]{Z}), it would imply that $n$ points and $n$ irreducible algebraic curves in the plane determine $O(n^{7/5+\eps})$ point-curve incidences.

\medskip

\noindent {\bf Approximate tangency and Besicovitch-Rado-Kinney sets}.
A Besicovitch-Rado-Kinney (BRK) set is a Borel subset of $\RR^2$ that contains a circle of every radius $1\leq r\leq 2$. Somewhat surprisingly, Besicovitch and Rado \cite{BR} and Kinney \cite{K} showed that there exist BRK sets that have Lebesgue measure 0. In \cite{W}, Wolff proved that every BRK set must have Hausdorff dimension 2. The study of BRK sets is closely related to both the Kakeya problem and the behavior of solutions to the wave equation in $2+1$ dimensions. See Wolff's survey \cite{W2} for an overview of the problem and its connections to harmonic analysis. 

To prove that every BRK set must have Hausdorff dimension 2, Wolff recast the question as a (discretized) incidence geometry problem. We will briefly summarize Wolff's reformulation here. Let $\delta>0$ be a small parameter and let $\mathcal{C}$ be a set of circles in the plane, each of radius between 1 and 2, with the property that each pair of circles from $\mathcal{C}$ have radii that differ by at least $\delta$. For each $C\in\mathcal{C}$, let $C^{\delta}$ be the $\delta$-neighborhood of $C$. The quantity
\begin{equation}\label{circleMaximalFnBd}
\int_{\RR^2}\Big(\sum_{C\in\mathcal{C}}\chi_{C^\delta}\Big)^{3/2}
\end{equation}
is closely related to the number of ``approximate'' tangencies amongst the circles in $\mathcal{C}$. Indeed, if two circles $C_1$ and $C_2$ intersect transversely, then $C_1^\delta\cap C_2^{\delta}$ has area roughly $\delta^2$. On the other hand, if $C_1$ and $C_2$ are approximately tangent and have very different radii, then $C_1^\delta\cap C_2^{\delta}$ has area roughly $\delta^{3/2}$ (which is much larger than $\delta^2$). Thus the expression \eqref{circleMaximalFnBd} can be thought of as counting the number of approximate tangencies amongst the circles in $\mathcal{C}$.

Wolff adapted the combinatorial incidence bounds of Clarkson, Edelsbrunner, Guibas, Sharir, and Wetzl from \cite{CEGSW} to prove that every set of $n$ circles in the plane determine $O(n^{3/2+\eps})$ points of ``approximate'' tangency, and this allowed him to control the size of expressions such as \eqref{circleMaximalFnBd}.

One can generalize the notion of a BRK set to other classes of curves, and it is conjectured that these generalized BRK sets must also have Hausdorff dimension 2. For curves that behave like pseudocircles in a certain geometric sense, this was proved by the author in \cite{Z2}. For other types of curves, however, the problem remains open. One potential way to analyze generalized BRK sets is to consider analogues of \eqref{circleMaximalFnBd} where the set of circles $\mathcal{C}$ is replaced by a different set of curves. Note that two distinct circles can only be tangent to order one. If $\mathcal{C}$ is replaced by a set of curves that can be tangent to order $k_0$, then the exponent in \eqref{circleMaximalFnBd} should be changed from $3/2$ to $(k_0+2)/(k_0+1)$. We are then faced with the problem of bounding the number of ``approximate'' $k$-th order tangencies amongst the curves in $\mathcal{C}$ for each $k=1,\ldots,k_0$, and we would like to show that there are $O(n^{(k+2)/(k+1)+\eps})$ such tangencies. Theorem \ref{mainThm} is a simple model case for this type of problem.

\section{Preliminaries}
In this section we will introduce the main tools needed to prove Theorem \ref{mainThm}. As in the introduction, all implicit constants many depend on $k$ (the order of tangency), and on $D$ (the maximum degree of the curves we are considering). 

Our proof will use the ``lifting'' method developed by Ellenberg, Solymosi, and the author in \cite{ESZ}. The basic idea is to lift a plane curve $\gamma\subset\CC^2$ to a space curve $L_k(\gamma)\subset\CC^{2+k}$. The curve $L_k(\gamma)$ will have the property that if $(x,y,z_1,\ldots,z_k)\in L_k(\gamma)$, then $(x,y)\in\gamma$ and the numbers $z_1,\ldots,z_k$ describe the $k$-th order tangency data of $\gamma$ at $(x,y)$. 

\begin{defn}
Let $\gamma\subset\CC^2$ be an irreducible curve that is not a vertical line, and let $k\geq 1$ be an integer. A curve $\beta\subset\CC^{2+k}$ is called a $k$-th order lift of $\gamma$ if it satisfies the following property: 

Let $(x,y)\in\gamma$ be a smooth point where $\gamma$ does not have vertical tangent, let $U$ be a neighborhood of $(x,y)$, let $V$ be a neighborhood of $x$, and let $g\colon V\to\CC$ be a function so that 
$$
\gamma\cap U = \{(t,\tilde y)\in U\colon \tilde y = g(t)\}.
$$ 
Then for every $(t,\tilde y)\in \gamma\cap U$, we have 
\begin{equation}
(t,\tilde y, z_1,\ldots, z_k)\in \beta\quad \textrm{if and only if}\quad z_j = \frac{d^j}{dt^j}g(t),\ j=1,\ldots,k.
\end{equation}
\end{defn}

\begin{lem}[Lifting]
Let $\gamma\subset\CC^2$ be an irreducible curve that is not a vertical line, and let $k\geq 1$ be an integer. Then there exists an irreducible curve $L_{k}(\gamma)\subset\CC^{2+k}$ of degree $O(1)$ that is a $k$-th order lift of $\gamma$.
\end{lem}
\begin{proof}
Let $f\in\CC[x,y]$ be an irreducible polynomial with $Z(f) = \gamma$. Consider $y$ (locally) as a function of $x$, and implicitly differentiate $\frac{d}{dx}f(x,y)$ $k$ times; we obtain the $k$ polynomial equations $P_1(x,y,y^\prime)=0,\ P_2(x,y,y^{\prime},y^{\prime\prime})=0,\ldots, P_k(x,y,y^\prime,\ldots,y^{(k)})=0$. For example, if $f(x,y)=x^2 + y^2 - 1$ and if $k=2$, we obtain the equations $0=P_1(x,y,y^\prime) = 2x + 2yy^\prime$ and $0=P_2(x,y,y^\prime, y^{\prime\prime}) = 2 + 2(y^\prime)^2 + 2yy^{\prime\prime}$. 

Define 
\begin{equation*}
\begin{split}
\tilde L_k(\gamma) = \{(x,y,z_1,\ldots,z_k)\in\CC^{2+k}\colon &f(x,y) = 0,\ P_1(x,y,z_1)=0,\\
&\quad P_2(x,y,z_1,z_2)=0,\ldots,P_k(x,y,z_1,\ldots,z_k)=0\}.
\end{split}
\end{equation*}

The polynomial $f(x,y)$ is non-zero, and for each $j=1,\ldots,k$, the polynomial $P_j(x,y,z_1,\ldots,z_j)$ is of the form $z_j Q_j(x,y,z_1,\ldots,z_{j-1})+R_j(x,y,z_1,\ldots,z_{j-1})$, where $Q_j$ is non-zero. In particular, this means that  $\tilde L_k(\gamma) $ is a proper intersection of $k+1$ hypersurfaces in $\CC^{2+k}$, so it is a (possibly reducible) algebraic space curve of degree $O(1)$.

Let $\pi\colon \CC^{2+k}\to\CC^2$ be the projection to the $xy$ plane. The projection of each irreducible component of $\tilde L_k(\gamma) $ is either Zariski dense in $\gamma$ or is a union of finitely many points. If $(x,y)\in\gamma$ is a smooth point with non-vertical tangent, then by the implicit function theorem there exists a neighborhood $U$ of $(x,y)$, a neighborhood $V$ of $x$, and a function $g\colon V\to\CC$ so that 
$$
\gamma\cap U = \{(t,\tilde y)\in V\colon \tilde y=g(t)\}.
$$ 
Then if $(t,\tilde y)\in U$, we have $(t,\tilde y,z_1,\ldots z_k)\in \tilde L_k(\gamma)$ if and only if $z_j= \frac{d^j}{dt^j}g(t),\ j=1,\ldots,k$. In particular, if $\gamma$ is not a vertical line, then the fiber of the projection $\pi\colon \tilde L_k(\gamma)\to\gamma$ above a generic point of $\gamma$ has cardinality one. This implies that there exists a unique irreducible component of $\tilde L_k(\gamma)$ whose projection is dense in $\gamma$. Call this component $L_k(\gamma)$. We have already established that $L_k(\gamma)$ has the claimed properties. 
\end{proof}

We will also need the following two elementary results from complex analysis. 

\begin{thm}[Holomorphic implicit function theorem]\label{holoImplicitThm}
Let $U\subset\CC^{2+k}$ be open and let $f\colon U\to\CC$ be holomorphic. Let $(x,y,z_1,\ldots,z_k)\in U$ and suppose $f(x,y,z_1,\ldots,z_k)  = 0$. If $\frac{d}{dz_k}f(x,y,z_1,\ldots,z_k)\neq 0$, then there exists an open set $V\subset\CC^{2+(k-1)}$ containing $(x,y,z_1,\ldots,z_{k-1})$, a holomorphic function $g\colon V\to\CC$, and a neighborhood $W\subset U$ of $(x,y,z_1,\ldots,z_k)$ so that
$$
\{ (\tilde x, \tilde y, \tilde z_1,\ldots,\tilde z_k)\in W\colon f(\tilde x, \tilde y, \tilde z_1,\ldots,\tilde z_k) = 0  \}=\{ (\tilde x, \tilde y ,\tilde z_1\ldots,\tilde z_{k})\in W\colon \tilde z_k = g(\tilde x, \tilde y,\tilde z_1,\ldots,\tilde z_{k-1})  \}.
$$
\end{thm}

\begin{thm}[Nonlinear Cauchy-Kowalevski theorem for ODE]\label{CKThm}
Let $U\subset\CC^{2+(k-1)}$ be a neighborhood of the point $(x,y,z_1,\ldots,z_{k-1})$. Let $g\colon U\to\CC$ be holomorphic. Then there exists a neighborhood $V$ of $x$ so that there is a unique function $h\colon V\to\CC$ that satisfies the Cauchy problem
$$
\left\{
\begin{array}{l}h^{(k)}(t) = g(t, h(t), h^\prime(t),\ldots, h^{(k-1)}(t))\ \textrm{for all}\ t\in V,\\
h(x) = y,\\
h^{(j)}(x) = z_j,\ j = 1,\ldots,k-1.
\end{array}
\right.
$$
\end{thm}

Next we will establish several results that connect the behavior of the curve $\gamma$ and its lift $L_k(\gamma)$.
\begin{lem}\label{curvesSameLem}
Let $P\in\CC[x,y,z_1,\ldots,z_k]$. Let $\gamma,\tilde\gamma$ be irreducible plane curves. Let $k\geq 1$ be an integer and let $(x,y,z_1,\ldots,z_k)\in\CC^{2+k}$. Suppose that 
\begin{itemize}
\item $(x,y,z_1,\ldots,z_k)\in L_k(\gamma)\subset Z(P)$.
\item $(x,y,z_1,\ldots,z_k)\in L_k(\tilde\gamma)\subset Z(P)$.
\item $(x,y)$ is a smooth point of $\gamma$ and $\tilde\gamma$ where neither curve has vertical tangent.
\item $\frac{d}{dz_k}P(x,y,z_1,\ldots,z_k)\neq 0$.
\end{itemize}
Then $\gamma=\tilde\gamma$.
\end{lem}
\begin{proof}
By the implicit function theorem (Theorem \ref{holoImplicitThm}), there exists a neighborhood $U$ of $(x,y,z_1,\ldots,z_{k-1})$ and a holomorphic function $g\colon U\to\CC$ with $g(x,y,z_1,\ldots,z_{k-1})=z_k$ and 
$$
P(\tilde x,\tilde y,\tilde z_1,\ldots,\tilde z_{k-1}, g(\tilde x, \tilde y, \tilde z_1,\ldots,\tilde z_{k-1}))=0\ \textrm{for all}\ (\tilde x, \tilde y, \tilde z_1,\ldots,\tilde z_{k-1})\in U.
$$ 
Again by the implicit function theorem, there is a neighborhood $V$ of $x$ and  functions $h(t),\tilde h(t)\colon V \to\CC$ so that $(t,h(t))$ (resp. $(t,\tilde h(t))$) is a parameterization of $\gamma$ (resp. $\tilde\gamma$) in a neighborhood of $(x,y)$. 

Thus for all $t\in V$, we have
$$
\big(t, h(t), h^\prime(t),\ldots, h^{(k)}(t)\big)\in L_k(\gamma),
$$
and 
$$
\big(t, h(t), h^\prime(t),\ldots, h^{(k-1)}(t)\big)\in U.
$$ 
Since $L_k(\gamma)\subset Z(P)$, we have 
$$
h^{(k)}(t) = g\big(t, h(t), h^\prime(t),\ldots, h^{(k-1)}(t)\big)\ \textrm{for all}\ t\in V,
$$
i.e. the function $h(t)$ satisfies the Cauchy problem
$$
\left\{
\begin{array}{l}h^{(k)}(t) = g(t, h(t), h^\prime(t),\ldots, h^{(k-1)}(t))\ \textrm{for all}\ t\in V,\\
h(x) = y,\\
h^{(j)}(x) = z_j,\ j=1,\ldots,k-1.\\
\end{array}
\right.
$$

On the other hand, the function $\tilde h(t)$ satisfies the same Cauchy problem. By Theorem \ref{CKThm}, we conclude that there exists a neighborhood $W\subset V$ of $x$ so that $h(t) = \tilde h(t)$ for all $t\in W.$ Thus $\gamma = \tilde\gamma$. 
\end{proof}

If distinct curves $\gamma$ and $\gamma^\prime$ satisfy the first three hypotheses of Lemma \ref{curvesSameLem}, then the fourth hypothesis must fail. We will record this observation as the following corollary.

\begin{cor}\label{cor1}
Let $\gamma$ and $\tilde\gamma$ be distinct irreducible curves in $\CC^2$, let $P\in\CC[x,y,z_1,\ldots,z_k],$ and suppose that $L_k(\gamma)\subset Z(P)$ and $L_k(\tilde\gamma)\subset Z(P)$. Let $(x,y,z_1,\ldots,z_k)\in L_k(\gamma)\cap L_k(\tilde\gamma)$. Suppose that $(x,y)$  is a smooth point of $\gamma$ and $\tilde\gamma$ where neither curve has vertical tangent. Then 
$$
(x,y,z_1,\ldots,z_k) \in Z( Q );\quad Q = \frac{d}{dz_k}P.
$$   
\end{cor}

Since the curve $L_k(\gamma)$ has degree $O(1)$, if $L_k(\gamma)$ intersects $Z(Q)$ in more than $O(\deg P)$ places then by B\'ezout's theorem $L_k(\gamma)$ must be contained in $Z(Q)$. We will record this observation as the following corollary.

\begin{cor}\label{intersectionForcesVertGradient}
Let $\mathcal{C}$ be a set of irreducible algebraic curves in $\CC^2$, each of degree at most $D$ and none of which are a vertical line. Let $k\geq 1$ be a positive integer, let $P\in\CC[x,y,z_1,\ldots,z_{k}]$, and suppose that $L_k(\gamma)\subset Z(P)$ for each $\gamma\in\mathcal{C}$. Then for each curve $\gamma\in\mathcal{C}$ we have that either
$$
|\{ L_k(\gamma)\cap L_k(\tilde\gamma)\colon\tilde\gamma\in\mathcal{C},\ \tilde\gamma\neq\gamma \}| = O(\deg P),
$$
or 
$$
L_k(\gamma) \subset Z( Q );\quad Q = \frac{d}{dz_k}P.
$$
\end{cor}

\section{Proof of Theorem \ref{mainThm}}
We are now ready to prove Theorem \ref{mainThm}. Our proof will use ``polynomial method'' ideas originally developed by Dvir \cite{D} and Guth and Katz \cite{GK}. The specific formulation used here is closely related to the arguments used by Kaplan, Sharir, and Shustin in \cite{KSS} to solve the joints problem in $\mathbb{R}^d$.

\begin{proof}[Proof of Theorem \ref{mainThm}] Applying a rotation if necessary, we can assume that none of the curves in $\mathcal{C}$ are vertical lines and that no two curves in $\mathcal{C}$ are tangent at a point of vertical tangency.

For each subset $\tilde{\mathcal{C}}\subset \mathcal{C}$, define
$$
\mathcal{P}_2(L_k(\tilde{\mathcal{C}}))=\bigcup_{\substack{\gamma,\tilde\gamma\in\tilde{\mathcal{C}}\\\gamma\neq\tilde\gamma}}L_k(\gamma)\cap L_k(\tilde\gamma).
$$

Observe that
$$
\sum_{p\in\CC^2}m_{k,\mathcal{C}}(p) = \sum_{\gamma\in\mathcal{C}} |\mathcal{P}_2(L_k(\mathcal{C}))\cap L_k(\gamma)|.
$$
Let $A$ be a large constant depending only on $D$ and $k$. Define $\mathcal{C}_0=\mathcal{C}$. For each $j=1,\ldots,$ define 
$$
\mathcal{C}_j = \big\{\gamma\in \mathcal{C}_{j-1}\colon |\mathcal{P}_2(L_k(\mathcal{C}_{j-1}))\cap L_k(\gamma)|\geq An^{\frac{1}{k+1}}\big\}.
$$
With this definition, we obtain an infinite sequence of nested sets $\mathcal{C}_0\supset\mathcal{C}_1\supset\mathcal{C}_2\supset\ldots.$ Let $N$ be the smallest index so that $\mathcal{C}_N = \mathcal{C}_{N+1}$; we have $N\leq n$. It might be the case that $\mathcal{C}_N = \emptyset$. Observe that
$$
\sum_{\gamma\in\mathcal{C}} |\mathcal{P}_2(L_k(\mathcal{C}))\cap L_k(\gamma)|\leq 2 \sum_{j=0}^{N-1} \sum_{\gamma\in\mathcal{C}_{j}\backslash\mathcal{C}_{j+1}}|\mathcal{P}_2(L_k(\mathcal{C}_j))\cap L_k(\gamma)|+\sum_{\gamma\in\mathcal{C}_N}|\mathcal{P}_2(L_k(\mathcal{C}_N))\cap L_k(\gamma)|.
$$
The first (double) sum on the right contains at most $|\mathcal{C}|=n$ terms, each of which have size at most $An^{\frac{1}{k+1}}$. Thus the sum has size $O(n^{\frac{k+2}{k+1}})$. To complete the proof, it suffices to show that the second sum has size $O(n^{\frac{k+2}{k+1}})$. We will show the stronger statement that $|\mathcal{C}_N| = O(n^{\frac{1}{k+1}})$ (this in fact implies that $\mathcal{C}_N=\emptyset$ if the constant $A$ is chosen sufficiently large, though we will not need this fact).

Let $P_{k}\in\CC[x,y,z_1,\ldots,z_k]$ be a non-zero polynomial of minimal degree that vanishes on the curves $\{L_k(\gamma)\colon\gamma\in\mathcal{C}_N\}$. We have $\deg P_k =O(|\mathcal{C}_N|^{\frac{1}{k+1}})=O(n^{\frac{1}{k+1}}).$ By Corollary \ref{intersectionForcesVertGradient}, we have that if $A$ is chosen sufficiently large then $L_k(\gamma)\subset Z(Q_k),\ Q_k = \frac{d}{dz_k}P_k$, for each $\gamma\in\mathcal{C}_N$. Since $\deg Q_k< \deg P_k,$ we conclude that $Q_k=0$, i.e. $P_{k}(x,y,z_1,\ldots,z_k) = P_{k-1}(x,y,z_1,\ldots,z_{k-1})$ for some non-zero polynomial $P_{k-1}\in\CC[x,y,z_1,\ldots,z_{k-1}]$ of degree $\deg P_{k-1} = \deg P_k$. 

Observe that $P_{k-1} $ is a polynomial of minimal degree in $\CC[x,y,z_1,\ldots,z_{k-1}]$ that vanishes on the curves $\{L_{k-1}(\gamma)\colon\gamma\in\mathcal{C}_N\}$; indeed, if there was a polynomial $R(x,y,z_1,\ldots,z_{{k-1}})$ of smaller degree that vanished on the curves $\{L_{k-1}(\gamma)\colon\gamma\in\mathcal{C}_N\}$, then the polynomial $\tilde R(x,y,z_1,\ldots,z_k) = R(x,y,z_1,\ldots,z_{k-1})$ would contradict the requirement that $P_k$ is a non-zero polynomial of minimal degree that vanishes on the curves $\{L_{k}(\gamma)\colon\gamma\in\mathcal{C}_N\}$.

Repeating the above argument, we see that each of the curves $L_{k-1}(\gamma),\ \gamma\in\mathcal{C}_N$ is contained in $Z(Q_{k-1}),\ Q_{k-1}=\frac{d}{dz_{k-1}}P_{k-1},$ and thus $Q_{k-1}=0$, so $P_{k-1}(x,y,z_1,\ldots,z_{k-1}) = P_{k-2}(x,y,z_1,\ldots,z_{k-2})$. Iterating this process $k$ times, we obtain a polynomial $P_0\in\CC[x,y]$ of degree $O(n^{\frac{1}{k+1}})$ that vanishes on each of the curves from $\mathcal{C}_N$. We conclude that $|\mathcal{C}_N| = O(n^{\frac{1}{k+1}})$. 
\end{proof}

\section{Further Remarks}
In \cite{ESZ}, Ellenberg, Solymosi, and the author proved a bound on the number of (first order) tangencies determined by a collection of plane algebraic curves in $F^2$, where $F$ is an arbitrary field whose characteristic is not too small compared to the number of curves. 

We conjecture that a similar result should hold for higher order tangencies. First, we must define what it means for two curves in $F^2$ to be tangent at a point $p\in F^2$ to order $\geq k$. One way of doing this is as follows. If $F$ is an algebraically closed field and $\gamma,\tilde\gamma$ are algebraic curves in $F^2$, we say that $\gamma$ and $\tilde\gamma$ are tangent at the origin to order $\geq k$ if $\dim_{F}F[[x,y]]/(f,\tilde f)\geq k$, where $f$ (resp.~$\tilde f)$ is a square-free polynomial whose zero-locus is $\gamma$ (resp.~$\tilde\gamma$) . We can extend this definition to define tangency at an arbitrary point $p\in F^2$ by translating $p$ to the origin. We can then extend this definition to the case where the field $F$ is not algebraically closed by replacing $F$ with its algebraic closure $\bar F$, and replacing $\gamma,\tilde\gamma$ with their Zariski closures inside $\bar F$. 

While this definition is rather technical, it simplifies considerably in the special case where $\gamma$ and $\tilde\gamma$ are the graphs of univariate polynomials. If 
\begin{equation*}
\begin{split}
\gamma &= \{(x,y)\in F^2\colon y = f(x)\},\\
 \tilde\gamma & = \{(x,y)\in F^2\colon y = \tilde f(x)\},
 \end{split}
 \end{equation*} 
 then $\gamma$ and $\tilde\gamma$ are tangent to order $\geq k$ at $(x,y)$ if $f(x) =\tilde f(x)$, and $\frac{d^j}{dx^j}f(x) = \frac{d^j}{dx^j}\tilde f(x)$ for each $j=1,\ldots,k$, where $\frac{d^j}{dx^j}f(x)$ is the (formal) derivative of $f$ at $x$. 

With this definition of tangency, the quantity $m_{k,\mathcal{C}}(p)$ from Definition \ref{countOfKTangencies} makes sense, and we can state our conjecture precisely.

\begin{conj}\label{arbFieldConj}
Let $F$ be a field, let $k\geq 1$ be an integer, and let $\mathcal{C}$ be a set of $n$ irreducible curves in $F^2$. Suppose that $n\leq \operatorname{char}(F)^{k+1}$ (if $\operatorname{char}(F) = 0$ then we place no restrictions on $n$). Then

\begin{equation}\label{boundOnNumberTangenciesArbField}
\sum_{p\in F^2}m_{k,\mathcal{C}}(p)=O(n^{\frac{k+2}{k+1}}).
\end{equation}
\end{conj}

If Conjecture \ref{arbFieldConj} is true, then it is sharp. To see this, let $\mathcal{C}^\prime=\mathcal{C}_1\cup\mathcal{C}_2$, where 
$$
\mathcal{C}_i=\{(x,y)\in \FF_p^2\colon y = ix^{k+1}+a_kx^k+\ldots+a_0; a_0,\ldots,a_k\in\FF_p\}.
$$
We have $|\mathcal{C}^\prime|=2p^{k+1};$ each curve in $\mathcal{C}^\prime$ is irreducible; and for each $(x,y,z_1,\ldots,z_k)\in\FF_p^{k+2}$, there are curves $\gamma_1=\{y=x^{k+1}+P_1(x)\}\in\mathcal{C}_1$ and $\gamma_2=\{y=2x^{k+1}+P_2(x)\}\in\mathcal{C}_2$ so that $ix^{k+1}+P_j(x) = y$, and $\frac{d^j}{dx^j}(ix^{k+1}+P_j(x))=z_j$ for each $i=1,2$ and each $j=1,\ldots,k$. Thus
$$
\sum_{p\in \FF_p^2}m_{k,\mathcal{C}^\prime}(p)=p^{k+1}.
$$

 Finally, let $\mathcal{C}\subset\mathcal{C}^\prime$ be obtained by randomly selecting each curve $\gamma\in\mathcal{C}^\prime$ with probability $1/4$. Then with high probability we have $|\mathcal{C}|\leq \frac{1}{2}|\mathcal{C}^\prime|=p^{k+1},$ and 
$$
\sum_{p\in \FF_p^2}m_{k,\mathcal{C}}(p)\geq \frac{1}{100}p^{k+1}\geq \frac{1}{100}|\mathcal{C}|^{(k+2)/(k+1)}.
$$


\begin{thebibliography}{}
%
\bibitem{BR} A.~Besicovitch and R.~Rado. A plane set of measure zero containing circumferences of every radius. \emph{J. London Math. Soc.} 717--719, 1968.
%
\bibitem{CEGSW} K.~L.~Clarkson, H.~Edelsbrunner, L.~J.~Guibas, M.~Sharir, and E.~Welzl. Combinatorial complexity bounds for arrangements of curves and spheres. \emph{Discrete Comput. Geom.} 5:99--160, 1990.
%
\bibitem{D} Z.~Dvir. On the size of Kakeya sets in finite fields. \emph{J. Amer. Math. Soc.} 22: 1093--1097, 2009. 
%
\bibitem{ESZ} J.~Ellenberg, J.~Solymosi, and J.~Zahl. New bounds on curve tangencies and orthogonalities. \emph{Discrete Analysis} 18, 2016. 
%
\bibitem{GK} L.~Guth and N.~Katz. Algebraic methods in discrete analogs of the Kakeya problem. \emph{Adv. Math.} 225(1): 2828--2839, 2010.
%
\bibitem{KSS}  H.~Kaplan, M. Sharir, and E. Shustin. On lines and joints. \emph{Discrete Comput. Geom.} 44: 838--843, 2010.
%
\bibitem{K} J.R.~Kinney. A thin set of circles. \emph{Amer. Math. Monthly} 75:1077--1081, 1968.
%
\bibitem{SZ} M.~Sharir and J.~Zahl. Cutting Algebraic Curves into Pseudo-segments and Applications. \emph{J. Combin. Theory Ser. A.} 150: 1-–35, 2017.
%
\bibitem{W} T.~Wolff. A Kakeya-type problem for circles. \emph{Am. J. Math.} 119: 985--1026, 1997.
%
\bibitem{W2} T.~Wolff. Recent work connected to the Kakeya problem. In \emph{Prospects in Mathematics}, 129--162, Amer. Math. Soc., Providence, RI 1999. 
%
\bibitem{Z} J.~Zahl. A Szemer\'edi-Trotter type theorem in $\mathbb{R}^4$. \emph{Discrete. Comput. Geom.} 54:513--572, 2015.
%
\bibitem{Z2} J.~Zahl. On the Wolff circular maximal function. \emph{Illinois J. Math.} 56:1281--1295, 2012. 
%
\end{thebibliography}
\end{document}